\documentclass[b5paper,11pt,onecolumn,twoside,UTF8]{article}
\usepackage[
	]{ctex}
\usepackage{NCMLEN}     %Load the style file.

\setboolean{FinalVersion}{true}
\newcommand{\mytitle}
{Proof Theory for Non-Contingency Logic}

\newcommand{\myrunningtitle}
{Proof Theory for Non-Contingency Logic}

\newcommand{\myfirstauthor}
{Borja Sierra Miranda}											  % Modify this only after your paper is accepted.

\newcommand{\myfirstaffiliation}{
{Institut f\"ur Informatik, Universit\"at Bern}
}

\newcommand{\myfirstemail}
{ borja.sierra@unibe.ch}
\renewcommand{\mysecondauthor}{Yunsong Wang}
\newcommand{\mysecondaffiliation}{
{Philosophy Department, Peking University}
}

\newcommand{\mysecondemail}
{yunsong.wang@pku.edu.cn}
\renewcommand{\mythirdauthor}{Lukas Zenger}
\newcommand{\mythirdaffiliation}{
{Philosophy Department, Peking University}
}

\newcommand{\mythirdemail}
{lukas.zenger@pku.edu.cn}

\newcommand{\myfourthauthor}{null}

\newcommand{\myfourthaffiliation}{
{Department 4, University 4}
}
\newcommand{\myfourthemail}
{xxxx@xxxx.xxx}

\newcommand{\myfifthauthor}{null}

\newcommand{\myfifthaffiliation}{
{Department 5, University 5}
}
\newcommand{\myfifthemail}
{ xxxx@xxxx.xxx}

\newcommand{\mythanks}{null}
\newcommand{\myabstract}
{Non-contingency logic $(\KWL)$  replaces the usual necessity operator of modal logic with an operator expressing that a proposition is necessarily true or necessarily false. Besides its intrinsic logical interest, it admits natural interpretations as knowing whether in epistemic logic and as decidability under the arithmetical interpretation of provability logic. Although the semantics of non-contingency logic have been extensively studied, its proof theory remains comparatively underdeveloped. In this paper, we develop a uniform proof-theoretic framework for $\KWL$ over a broad class of frame conditions. Our approach is based on generalized path conditions (GPCs), a grammar-theoretic formalism that uniformly captures many standard modal frame properties. For every finite set $\gpc$ of GPCs, we construct a corresponding labelled sequent calculus. All calculi share a common set of logical rules and differ only by a single structural rule generated from $\gpc$, which captures the underlying frame conditions. We prove that these calculi are sound and complete with respect to their corresponding frame classes, thereby providing a uniform proof theory for a large family of non-contingency logics.}
\newcommand{\mykeywords}
{non-contingency $\cdot$ non-normal modal logic $\cdot$ proof theory $\cdot$ labelled sequent calculus $\cdot$ generalized path condition}

\usepackage[pdfborder=0]{hyperref}
\usepackage{latexsym}
\usepackage{graphicx}
\usepackage{xcolor}

\usepackage{graphicx} % Required for inserting images
\usepackage{tikz}
\usetikzlibrary{arrows,calc,patterns,positioning,shapes}
   \usetikzlibrary{decorations.pathmorphing}
   \tikzset{
       modal/.style={>=stealth',shorten >=1pt,shorten <=1pt,auto,
                     node distance=1.5cm,semithick},
       world/.style={circle,draw,minimum size=1cm,fill=gray!15},
       point/.style={circle,draw,fill=black,inner sep=0.5mm},
       reflexive/.style={->,in=120,out=60,loop,looseness=#1},
       reflexive/.default={5},
       reflexive point/.style={->,in=135,out=45,loop,looseness=#1},
       reflexive point/.default={25},
       }

 \tikzset{
       reflexive above/.style={->,loop,in=120,out=60,looseness=#1},
       reflexive above/.default={7},
       reflexive below/.style={->,loop,in=240,out=300,looseness=#1},
       reflexive below/.default={7},
       reflexive left/.style={->,loop,in=150,out=210,looseness=#1},
       reflexive left/.default={7},
       reflexive right/.style={->,loop,in=30,out=330,looseness=#1},
       reflexive right/.default={7}
}
\usetikzlibrary{arrows.meta}
\usetikzlibrary{calc}
\usepackage{tikz-cd}
\usepackage{xfrac}
\usepackage{amsmath}
\usepackage{amssymb}
\usepackage{relsize}
\usepackage{mathtools}
\usepackage{bussproofs}
\usepackage{proof}
\usepackage{enumitem}
\usepackage{comment}
\usepackage{float}
\usepackage{amsthm}
\usepackage{mathtools}
\usepackage{mathrsfs} 
\usepackage{esvect,pgf, tikz, color}
\usetikzlibrary{arrows, automata, positioning}
\usepackage{stmaryrd}

{\def\scalefactor{#1}\begin{center}\proofSkipAmount \leavevmode}%
	{\scalebox{\scalefactor}{\DisplayProof}\proofSkipAmount \end{center} }

\newcommand{\K}{\mathop{\raisebox{-0.2ex}{\scalebox{0.85}{\rotatebox{90}{$\nabla$}}}}}

\newcommand{\fw}{\mathcal{M}}
\newcommand{\fv}{\mathcal{N}}
\newcommand{\W}{\mathcal{W}}
\newcommand{\V}{\mathcal{V}}
\newcommand{\R}{\mathcal{R}}
\newcommand{\Rel}{\mathrel{\R}}
\newcommand{\Prop}{\mathsf{Prop}}
\newcommand{\Var}{\mathsf{Var}}
\newcommand{\F}{\mathcal{F}}
\newcommand{\G}{\mathcal{G}}
\newcommand{\Fr}{\mathscr{F}}

\newcommand{\LKW}{\mathcal{L}_{\mathsf{NCL}}}
\newcommand{\KWL}{\mathsf{NCL}}
\newcommand{\KW}{\mathrm{NCL}}

\newcommand{\x}{\mathsf{x}}
\newcommand{\y}{\mathsf{y}}
\newcommand{\z}{\mathsf{z}}

\newcommand{\rel}{\mathrel{\mathsf{R}}}

\newcommand{\wt}{\alpha}

\newcommand{\xt}{x}

\newcommand{\xti}{\overline{x}}

\newcommand{\gpc}{\mathbf{G}}

\newif\ifcomments
\commentstrue   % set to \commentsfalse to hide

\begin{document}
\ifthenelse{\boolean{FinalVersion}}{}{\linenumbers}
\printtitlepage

\section{Introduction}

\noindent Modal logic traditionally studies the notion of necessity, represented by the operator $\Box$. Another fundamental modal notion is \emph{non-contingency} ($\K$). A statement $\varphi$ is non-contingent if it is necessarily true or necessarily false, and can therefore be defined in terms of necessity by $\K \varphi := \Box \varphi \vee \Box \neg \varphi$. 
Although definable in the language of normal modal logic, \emph{non-contingency logic} ($\KWL$), obtained by taking $\K$ instead of $\Box$ as primitive, has developed into an active research area in its own right~\cite{Humberstone1995,Zolin1999}. Besides its metaphysical interpretation, non-contingency admits several natural readings in applications. In epistemic logic, $\K\varphi$ expresses \emph{knowing whether} $\varphi$ holds rather than knowing that it holds, a notion that captures many practical reasoning tasks in which the objective is to resolve uncertainty rather than establish a particular outcome~\cite{FanWangDitmarsch_2015}.
Another important interpretation arises in arithmetic. Under the standard provability interpretation of Gödel--Löb logic (GL)~\cite{Solovay_1976}, the modal operator $\Box$ represents \emph{provability} in Peano Arithmetic ($\mathsf{PA}$). Consequently, $\K\varphi$ naturally expresses that $\varphi$ is \emph{decidable} in $\mathsf{PA}$, that is, either $\varphi$ or $\neg\varphi$ is provable.

% Since the pioneering work of Montgomery and Routley~\cite{MontgomeryRoutley_1966} in the 1960s, non-contingency logic has been studied extensively from both semantic and axiomatic perspectives. Important developments include results on frame definability and correspondence~\cite{Zolin1999}, a bisimulation theory~\cite{FanWangDitmarsch2014}, and complete axiomatizations over a range of standard frame classes~\cite{Kuhn1995,FanWangDitmarsch_2015,Fan2020FamilyK}. Non-contingency has also been studied under neighbourhood semantics~\cite{FanDitmarsch2015}. Fan~\cite{Fan2025Bundled} developed a uniform approach to axiomatization based on almost-definability schemas, applicable in particular to non-contingency logic.
% Zolin developed sound and complete sequent calculi for several frame classes, established Craig interpolation, and showed that cut cannot in general be eliminated~\cite{Zolin_2002}, while also giving a cut-based calculus for $\mathsf{GL}$-frames~\cite{Zolin_2001}. More recently, Venturi and Yago developed tableaux for essence and contingency~\cite{VenturiYago2021}. Petrukhin introduced a cut-free hypersequent calculus for non-contingency over S5-frames~\cite{Petrukhin_2022} and subsequently developed nested sequent calculi for several standard frame classes~\cite{Petrukhin2025}.

Since the pioneering work of Montgomery and Routley~\cite{MontgomeryRoutley_1966} in the 1960s, non-contingency logic has been studied extensively from both semantic and axiomatic perspectives. Important developments include results on frame definability and correspondence~\cite{Zolin1999}, a bisimulation theory~\cite{FanWangDitmarsch2014}, and complete axiomatizations over a range of standard Kripke frame classes~\cite{Kuhn1995,FanWangDitmarsch_2015,Fan2020FamilyK}. Non-contingency has also been studied under neighbourhood semantics~\cite{FanDitmarsch2015}. Fan~\cite{Fan2025Bundled} developed a uniform approach to axiomatization based on almost-definability schemas, applicable in particular to non-contingency logic. Zolin developed sound and complete sequent calculi for several frame classes, established Craig interpolation, and showed that cut cannot in general be eliminated~\cite{Zolin_2002}, while also giving a cut-based calculus for $\mathsf{GL}$-frames~\cite{Zolin_2001}. More recently, Venturi and Yago developed tableaux for essence and contingency~\cite{VenturiYago2021}. Petrukhin introduced a cut-free hypersequent calculus for non-contingency over S5-frames~\cite{Petrukhin_2022} and subsequently developed nested sequent calculi for several standard frame classes~\cite{Petrukhin2025}.

The main goal of this paper is to develop a uniform proof-theoretic framework for non-contingency logic. To this end, we present a single family of labelled sequent calculi parameterised by finite collections of \emph{generalized path conditions} (GPCs). Generalized path conditions provide a grammar-theoretic description of frame conditions and uniformly capture many standard properties of accessibility relations, including reflexivity, symmetry, transitivity, and Euclideanness. For every finite collection of GPCs, we obtain a labelled sequent calculus consisting of a fixed set of logical rules together with a single structural rule that saturates the relational component with respect to the chosen frame conditions. 
This yields a modular framework in which soundness and completeness are each established uniformly for the entire family of calculi: soundness by a standard induction and completeness by proof-search.
Our treatment of relational paths is inspired by propagation mechanisms for grammar logics with converse~\cite{TiuIanovskiGore2012} and by their development in refined labelled systems~\cite{Lyon_2021}. Unlike those normal modal systems, our object language takes non-contingency as primitive and therefore requires different logical rules.

Labelled calculi are particularly well suited to our approach because relational information is represented explicitly, allowing frame conditions to be handled by structural rules while leaving the logical rules unchanged. The resulting calculi support systematic proof-search and make it possible to treat a wide range of frame conditions within a single proof-theoretic framework. Overall, our results provide a uniform labelled-sequent framework for non-contingency logic over frame classes specified by finite sets of GPCs. \smallskip

\noindent \textbf{Outline.} The remainder of the paper is organized as follows. Section~\ref{sec: non-contingency logic} introduces non-contingency logic and generalized path conditions. Section~\ref{sec: sequent calculi} presents the labelled sequent calculi and proves soundness. Completeness is established in Section~\ref{sec: completeness} by means of proof-search. Section~\ref{sec: GL} concludes by discussing future work and by announcing further results about the proof-theoretic properties of our framework and about $\mathsf{GL}$ with non-contingency.

\section{Non-Contingency Logic}\label{sec: non-contingency logic}

\noindent The \emph{language} $\LKW$ of non-contingency logic is generated from a
countably infinite set $\Prop$ of propositions, the constant
$\bot$, implication $\to$, and the modal operator $\K$:
\[
\varphi ::= \bot \mid p \mid (\varphi\to\varphi) \mid \K\varphi,
\qquad p\in\Prop.
\]
We use $p,q,r,\ldots$ for propositional variables and
$\varphi,\psi,\gamma,\ldots$ for formulae. The formula $\K\varphi$ is read as ``$\varphi$ is non-contingent''. The remaining
Boolean connectives are defined as usual. We interpret formulae of $\KWL$ over standard Kripke models.
\iffalse
If complexity is needed later, define it by
\[
c(\bot)=c(p)=0,\qquad
c(\varphi\to\psi)=\max\{c(\varphi),c(\psi)\}+1,\qquad
c(\K\varphi)=c(\varphi)+1.
\]
For a finite set $\Gamma$, put
$c(\Gamma)\coloneqq\sum_{\varphi\in\Gamma}c(\varphi)$.
\fi
\begin{definition}
A \emph{frame} is a pair $\F=(\W,\R)$, where $\W$ is a nonempty set whose elements are called \emph{worlds} and $\R\subseteq\W\times\W$ is a binary relation called the \emph{accessibility relation}. A \emph{model} based on $\F$ is a triple
$\fw=(\W,\R,\V)$, equivalently $(\F,\V)$, where
$\V:\W\to 2^{\Prop}$ is a function called \emph{valuation}.
\end{definition}

We use $\F,\G$ for frames, $\fw,\fv$ for models, and
$w,v,u,\ldots$ for worlds. A frame is called \emph{reflexive} if $\R$ is reflexive; \emph{transitive} and \emph{symmetric} frames are defined analogously. An \emph{S4-frame} is
reflexive and transitive, and an \emph{S5-frame} is reflexive, transitive, and
symmetric. The \emph{truth relation} $\models$ is defined by the usual clauses for $\bot$, atoms,
and implication, together with
\[
\fw,w\models\K\varphi
\quad\Longleftrightarrow\quad
\bigl(\forall v\,(w\Rel v\Rightarrow\fw,v\models\varphi)\bigr)
\ \text{or}\
\bigl(\forall v\,(w\Rel v\Rightarrow\fw,v\not\models\varphi)\bigr).
\]
Given a model $\fw=(\W,\R,\V)$ and $w\in\W$, a formula $\varphi$ is
\emph{true at $w$} if ${\fw,w\models\varphi}$. It is \emph{true in
$\fw$}, written $\fw\models\varphi$, if it is true at every world of
$\fw$. A formula is \emph{satisfiable over} a class of frames $\Fr$ if
it is true at some world of some model based on a frame in $\Fr$; it is
\emph{unsatisfiable over $\Fr$} otherwise. It is \emph{valid over
$\Fr$} if it is true in every model based on a frame in $\Fr$ and
\emph{falsifiable over $\Fr$} otherwise. \smallskip

\noindent\textbf{Grammar-theoretic preliminaries.}
We present the preliminaries needed to define GPCs. We associate binary relations with words over the two-symbol alphabet $\Sigma\coloneqq\{\xt,\xti\}$. The symbols $\xt$ and $\xti$ represent
forward and backward traversal of the relation,
respectively. Let $\Sigma^*$ be the set of finite words over $\Sigma$, including the empty word $\epsilon$. Concatenation of words is written by juxtaposition, where $\epsilon\beta=\beta\epsilon=\beta$. We define the \emph{converse operation} on words by setting $\epsilon^{-1}\coloneqq\epsilon$, $\xt^{-1}\coloneqq\xti$ and $\xti^{-1}\coloneqq\xt$ and for
$\beta=z_1\cdots z_n$,
$\beta^{-1}\coloneqq z_n^{-1}\cdots z_1^{-1}$. The words $\beta$ and $\beta^{-1}$ are called \emph{converses}.

\begin{definition}\label{d:Sigma-system}
A \emph{$\Sigma$-system}\footnote{Note that $\Sigma$-systems are a restricted form of semi-Thue systems~\cite{Pos47}.} is a set $\mathcal S$ of productions
$z\mapsto\beta$, where $z\in\Sigma$ and $\beta\in\Sigma^*$.
\end{definition}
We are interested in a specific type of $\Sigma$-system. Namely for $\alpha\in\Sigma^*$, let
\[
\mathcal S_\alpha\coloneqq
\{\xt\mapsto\alpha,\ \xti\mapsto\alpha^{-1}\}.
\]
Thus in such $\Sigma$-systems every production is accompanied by its converse production.

\begin{definition}\label{d:semi-thue-deriv-lang}
Let $\mathcal S$ be a $\Sigma$-system. We write
$\beta\mapsto_{\mathcal S}\gamma$ if there are
$\rho,\tau\in\Sigma^*$ and a production $z\mapsto\eta\in\mathcal S$
such that $\beta=\rho z\tau$ and $\gamma=\rho\eta\tau$.
In this case, $\gamma$ is derived from $\beta$ in \emph{one step}. A \emph{derivation} of $\gamma$ from $\beta$, written $\beta \mapsto_{\mathcal{S}}^\ast \gamma$, is a
finite sequence of one-step derivations
\[
\beta=\delta_0\mapsto_{\mathcal S}\delta_1
\mapsto_{\mathcal S}\cdots
\mapsto_{\mathcal S}\delta_n=\gamma,
\]
whose length is $n$, where $\beta = \gamma$ for $n = 0$. Finally, put
$\mathcal L_{\mathcal S}(\beta)
\coloneqq\{\gamma\mid\beta\mapsto_{\mathcal S}^*\gamma\}$.
\end{definition}

\noindent \textbf{Generalized path conditions.}
Let $X$ be a set, which is allowed to be empty, and let
$Q\subseteq X\times X$. We associate a relation $Q_\beta$ on
$X$ with every word $\beta\in\Sigma^*$ by putting
\[
Q_\epsilon\coloneqq\{(a,a)\mid a\in X\},\qquad
Q_{\xt}\coloneqq Q,\qquad
Q_{\xti}\coloneqq Q^{-1},
\]
and
\[
Q_{\beta z}\coloneqq Q_z\circ Q_\beta
\qquad(\beta\in\Sigma^*,\ z\in\Sigma).
\]
Thus $(a,b)\in Q_{z_1\cdots z_n}$ precisely when there are
$a=a_0,\ldots,a_n=b$ such that
$(a_{i-1},a_i)\in Q_{z_i}$ for every $1\leq i\leq n$.
The domain $X$ is kept fixed throughout this construction.

% Let $\F = (\W, \R)$ be a frame. We define $\R_{\xt} \coloneqq \R$ and $\R_{\xti} \coloneqq \R^{-1} = \{(v,w) \mid (w,v) \in \R\}$. For the empty string $\epsilon$, we define $\R_{\epsilon} = \{(w,w) \mid w \in \W\}$. For a string $\wt\zt \in \Sigma^{*}$, we define $\R_{\wt\zt} \coloneqq \R_{\zt} \circ \R_{\wt}$ where $\circ$ is relation composition.\footnote{In other words, $\R_{\zt} \circ \R_{\wt} = \{(w,u) \mid \exists v \in \W, (w,v) \in \R_{\wt}, (v,u) \in \R_{\zt}\}$.} 

\begin{definition}
Let $\alpha\in\Sigma^*$. A frame $\F=(\W,\R)$ satisfies the
\emph{generalized path condition generated by $\alpha$}, denoted by
$\gpc(\alpha)$, if
\[
\R_\beta\subseteq\R
\qquad\text{for every }
\beta\in\mathcal L_{\mathcal S_\alpha}(\xt).
\]
For a finite set $\gpc$ of generalized path conditions, a frame is a
\emph{$\gpc$-frame} if it satisfies every member of $\gpc$. We write $\Fr_\gpc$ for the class of all $\gpc$-frames and put
\[
\mathcal S(\gpc)
\coloneqq
\bigcup_{\gpc(\alpha)\in\gpc}\mathcal S_\alpha.
\]
\end{definition}

\begin{lemma}[Word relations]\label{lem:word-relations}
Let $X$ be a set, let $Q,P\subseteq X\times X$, and let $\beta,\gamma\in\Sigma^*$.
\begin{enumerate}
\item $Q_{\beta\gamma}=Q_\gamma\circ Q_\beta$.
\item $Q_{\beta^{-1}}=(Q_\beta)^{-1}$.
\item If $Q\subseteq P$, then $Q_\beta\subseteq P_\beta$.
\item Suppose that $\delta=\rho z\tau$,
$\eta=\rho\gamma\tau$, and
$Q_\gamma\subseteq Q_z$, where $z\in\Sigma$.
Then $Q_\eta\subseteq Q_\delta$.
\end{enumerate}
\end{lemma}
\begin{proof}
The concatenation identity follows by induction on the length
of $\gamma$. The converse and monotonicity claims follow by
induction on the length of the relevant word. For the final
claim, use the concatenation identity twice to obtain
\[
Q_\eta=Q_\tau\circ Q_\gamma\circ Q_\rho
\subseteq
Q_\tau\circ Q_z\circ Q_\rho
=Q_\delta
\]
and monotonicity of relational composition.
\end{proof}
The paired productions in $\mathcal S(\gpc)$ entail the following
useful converse property:
\begin{equation}\label{eq:converse-derivation}
\beta\mapsto_{\mathcal S(\gpc)}^*\gamma
\quad\Longrightarrow\quad
\beta^{-1}\mapsto_{\mathcal S(\gpc)}^*\gamma^{-1}.
\end{equation}
Indeed, the inverse of an application of
$\xt\mapsto\alpha$ is an application of
$\xti\mapsto\alpha^{-1}$, and conversely. Induction on the
length of a derivation proves \eqref{eq:converse-derivation}.
\begin{lemma}[Finite basis]\label{lem:gpc-finite-basis}
Let $\F=(\W,\R)$ be a frame and let $\alpha\in\Sigma^*$.
Then
\[
\F\models\gpc(\alpha)
\quad\Longleftrightarrow\quad
\R_\alpha\subseteq\R.
\]
Consequently, $\F$ is a $\gpc$-frame if and only if
$\R_\alpha\subseteq\R$ for every
$\gpc(\alpha)\in\gpc$.
\end{lemma}
\begin{proof}
If $\F\models\gpc(\alpha)$, then
$\alpha\in\mathcal L_{\mathcal S_\alpha}(\xt)$, because
$\xt\mapsto\alpha$ is a production of $\mathcal S_\alpha$.
Hence $\R_\alpha\subseteq\R$.
Conversely, suppose that $\R_\alpha\subseteq\R$. By
Lemma~\ref{lem:word-relations},
\[
\R_{\alpha^{-1}}
=(\R_\alpha)^{-1}
\subseteq\R^{-1}
=\R_{\xti}.
\]
Thus both productions of $\mathcal S_\alpha$ replace a letter
by a word whose relation is contained in the relation of that
letter. Lemma~\ref{lem:word-relations}(4), followed by
induction on derivation length, gives
\[
\xt\mapsto_{\mathcal S_\alpha}^*\beta
\quad\Longrightarrow\quad
\R_\beta\subseteq\R_{\xt}=\R.
\]
This is exactly the condition $\gpc(\alpha)$.
The final assertion follows by applying this equivalence to
each member of $\gpc$.
\end{proof}

% \begin{example}
% By Lemma~\ref{lem:gpc-finite-basis}, the following
% correspondences are immediate:
% \[
% \begin{array}{rcl}
% \gpc(\epsilon) &\Longleftrightarrow&
% \R_\epsilon\subseteq\R \quad\text{(reflexivity)},\\
% \gpc(\xt\xt) &\Longleftrightarrow&
% \R^2\subseteq\R \quad\text{(transitivity)},\\
% \gpc(\xti) &\Longleftrightarrow&
% \R^{-1}\subseteq\R \quad\text{(symmetry)},\\
% \gpc(\xti\xt) &\Longleftrightarrow&
% \R\circ\R^{-1}\subseteq\R
% \quad\text{(right Euclideanness)}.
% \end{array}
% \]
% Note that seriality is not expressible by GPCs.
% \end{example}

\begin{example}
By Lemma~\ref{lem:gpc-finite-basis}, the following
correspondences are immediate:
\[
\setlength{\arraycolsep}{2pt}
\begin{array}{@{}rcl@{\quad}rcl@{}}
\gpc(\epsilon) &\Longleftrightarrow&
\R_\epsilon\subseteq\R \;\text{(reflexivity)},
& \gpc(\xt\xt) &\Longleftrightarrow&
\R^2\subseteq\R \;\text{(transitivity)},\\
\gpc(\xti) &\Longleftrightarrow&
\R^{-1}\subseteq\R \;\text{(symmetry)},
& \gpc(\xti\xt) &\Longleftrightarrow&
\R\circ\R^{-1}\subseteq\R
\;\text{(right Euclideanness)}.
\end{array}
\]
Note that seriality is not expressible by GPCs.
\end{example}

\begin{definition}
For a finite set $\gpc$ of GPCs, the logic $\mathbf{NCL}_\gpc$ is defined to be the set of formulae valid over $\Fr_\gpc$.
\end{definition}
Thus $\mathbf{NCL}_\emptyset$ is the logic of all frames. The logics of
reflexive, S4-, and S5-frames are, respectively, $\mathbf{NCL}_{\{\gpc(\epsilon)\}}$, $\mathbf{NCL}_{\{\gpc(\epsilon),\gpc(\xt\xt)\}}$ and $\mathbf{NCL}_{\{\gpc(\epsilon),\gpc(\xt\xt),\gpc(\xti)\}}$.

\begin{lemma}
For every finite set $\gpc$ of GPCs, the logic $\mathbf{NCL}_{\gpc}$ is non-normal:
\[
\K(p\to q)\to(\K p\to\K q)
\notin\mathbf{NCL}_\gpc.
\]
\end{lemma}
\begin{proof}
Let $\W=\{w,v,u\}$, let $\R=\W\times\W$, and let
$\fw=(\W,\R,\V)$, where
$\V(w)=\V(v)=\emptyset$ and $\V(u)=\{q\}$. 
Since $\R$ is universal, $\R_\beta\subseteq\R$ for every
$\beta\in\Sigma^*$; hence $(\W,\R)$ is a $\gpc$-frame for every
$\gpc$.
The atom $p$ is false everywhere, so $p\to q$ is true everywhere.
Thus $\fw,w\models\K(p\to q)$ and $\fw,w\models\K p$.
However, $q$ is false at $v$ and true at $u$, and both are
$\R$-successors of $w$. Hence $\fw,w\not\models\K q$.
\end{proof}

\section{Sequent Calculi}\label{sec: sequent calculi}

% \noindent This section introduces for every finite set $\gpc$ of GPCs a labelled sequent calculus $\KW_\gpc$. Each calculus consists of the same set of axioms and rules for implication and $\K$, as well as a single rule parametrized by $\gpc$ which captures the frame conditions expressed in $\gpc$.

\noindent This section introduces, for every finite set $\gpc$ of GPCs, a labelled sequent calculus $\KW_\gpc$. We follow the labelled-sequent methodology in which relational semantics is represented explicitly by labels and relational atoms~\cite{Negri2005}.

Let $\Var = \{\x, \y, \z, \ldots\}$ be a countably infinite set of variables. A \emph{relational atom} is an expression
$\x\rel\y$ with $\x,\y\in\Var$, and a finite set of relational
atoms is called a \emph{control}. A \emph{labelled formula} is an
expression $\x:\varphi$, where $\x\in\Var$ and $\varphi$ is a formula.
For a set $\Gamma$ of labelled formulae, let $\Var(\Gamma)$ be the set
of variables occurring as labels in $\Gamma$. Similarly, for a control $\Omega$,
let $\Var(\Omega)$ be the set of variables occurring in $\Omega$. Given a model $\fw=(\W, \R, \V)$, an \emph{assignment for $\fw$} is a map $\lambda: \Var \longrightarrow \W$. The \emph{interpretation} of labelled formulae and relational atoms is defined as follows:
\[
\begin{array}{rcl}
\fw,\lambda\Vdash\x\rel\y
&\Longleftrightarrow& \lambda(\x)\Rel\lambda(\y),\\
\fw,\lambda\Vdash\x:\varphi
&\Longleftrightarrow& \fw,\lambda(\x)\models\varphi.
\end{array}
\]
A \emph{sequent} is an ordered triple $\sigma =(\Omega,\Gamma,\Delta)$,
written $\Omega\dashv\Gamma\Rightarrow\Delta$, where $\Omega$ is a
control and $\Gamma,\Delta$ are finite sets of labelled formulae.
The pair $\Gamma\Rightarrow\Delta$ is the \emph{body}; $\Gamma$ and
$\Delta$ are its \emph{antecedent} and \emph{consequent}. %Since contexts are sets, exchange and contraction are implicit. 
For $\sigma=\Omega\dashv\Gamma\Rightarrow\Delta$, let
$\sigma_c=\Omega$, $\sigma_L=\Gamma$, $\sigma_R=\Delta$, and
$\Var(\sigma)=\Var(\Omega)\cup\Var(\Gamma)\cup\Var(\Delta)$.
\begin{definition}
A sequent $\sigma$ is \emph{true in $\fw$ under $\lambda$}, written
$\fw,\lambda\Vdash\sigma$, if $\fw,\lambda\Vdash\x \rel \y$ for every $\x \rel \y \in \sigma_c$ and $\fw,\lambda\Vdash \x: \varphi$ for every $\x: \varphi\in\sigma_L$ imply that $\fw,\lambda\Vdash \y: \psi$ for some $\y: \psi \in\sigma_R$. It is \emph{valid over} a frame class $\Fr$ if $\sigma$ is true in every
model based on a frame in $\Fr$ under every assignment; otherwise $\sigma$ is
\emph{falsifiable over $\Fr$}.
\end{definition}

\begin{definition}
Let $\gpc$ be a finite set of GPCs. The calculus $\KW_\gpc$
consists of the initial sequents and rules displayed in
Tables~\ref{tab:rules1} and~\ref{tab:rules2}.
\end{definition}
\begin{table}[t]
    \centering
    \small
    \begin{tabular}{|c c|}
    \hline
    & \\
      $\infer[\bot]{\Omega \dashv \Gamma, \x: \bot \Rightarrow \Delta}{}$ & $\infer[\mathsf{id}]{\Omega \dashv \Gamma, \x: \varphi \Rightarrow \x: \varphi, \Delta}{}$ \\

      & \\
      
      $\infer[\mathsf{{\to} L}]{\Omega \dashv \Gamma, \x: \varphi \to \psi \Rightarrow \Delta}{\Omega \dashv \Gamma \Rightarrow \x: \varphi, \Delta & \Omega \dashv \Gamma, \x: \psi \Rightarrow \Delta}$ & $\infer[\mathsf{{\to}R}]{\Omega \dashv \Gamma \Rightarrow \x: \varphi \to \psi, \Delta}{\Omega \dashv \Gamma, \x: \varphi \Rightarrow \x: \psi, \Delta}$ \\
      & \\
      \hline
    \end{tabular}
    \caption{Initial sequents and implication rules of $\KW_\gpc$.}
    \label{tab:rules1}
\end{table}

The implication rules and the initial sequents ($\bot$ and $\mathsf{id}$) are independent of the
control. The remaining rules govern the modality and the frame
conditions. Read bottom-up, the rules preserve falsifiability, and so $\K\mathsf L$ expresses that if $\x:\K\varphi$ is true in a model $\fw$ under an assignment $\lambda$, then $\varphi$ has one uniform truth value at all successors of $\lambda(\x)$. The two
premises represent these two possibilities. For $\K\mathsf R$, if $\x:\K\varphi$ is false in $\fw$ under $\lambda$, the rule introduces two distinct fresh variables, one representing a successor where $\varphi$ is true and the other where it is false. All rules apart from $\mathsf{r}_\gpc$ are called \emph{logical rules}. For each logical rule, the formula to which the rule is applied in the conclusion is called \emph{principal}.
The formulae introduced in the premises by decomposing the
principal formula are called its \emph{residuals}. All other
formulae are called \emph{side formulae}. The rule $\mathsf r_\gpc$ is called a \emph{structural rule} and replaces the relation encoded by the control with its least $\gpc$-closed extension on the fixed label domain of the conclusion. The rule $\mathsf{r}_\gpc$ has no principal formula; every formula
occurrence in an instance of this rule is a side formula. The formal definition of $\mathsf{r}_\gpc$ is given as follows.

\begin{definition}
Let $\sigma=\Omega\dashv\Gamma\Rightarrow\Delta$ be a sequent.  The
\emph{induced frame} $\F_\sigma=(\W,\R)$ of $\sigma$ is defined as
follows.
\begin{itemize}
  \item $\W\coloneqq\Var(\sigma)$;
  \item $\x\Rel\y$ if and only if $\x\rel\y\in\Omega$.
\end{itemize}
\end{definition}
Strictly speaking, if $\Var(\sigma)=\emptyset$, then $\F_\sigma$ is
an empty relational graph rather than a frame; nevertheless, all relational constructions below remain
well-defined in this degenerate case. The proof of the following lemma is immediate and omitted.
\begin{lemma}\label{lem: induced-frame}
If $\Var(\sigma)\neq\emptyset$, then the induced frame $\F_\sigma$
is a frame.
\end{lemma}

\begin{table}[t]
    \centering
    \small
    \begin{tabular}{|c|}
    \hline
    \\
      $\infer[\mathsf{{\K} L}]{\Omega \dashv \Gamma, \x: \K \varphi \Rightarrow \Delta}{\Omega \dashv \Gamma, \x: \K \varphi, \{\y: \varphi \mid \x \rel \y \in \Omega\} \Rightarrow \Delta & \Omega \dashv  \Gamma, \x: \K \varphi \Rightarrow \{\y: \varphi \mid \x \rel \y \in \Omega\}, \Delta}$ \\

       \\
    
     $\infer[\mathsf{{\K}R}]
{\Omega \dashv \Gamma \Rightarrow \x:\K\varphi,\Delta}
{\Omega,\x\rel\y,\x\rel\z
 \dashv
 \Gamma,\y:\varphi
 \Rightarrow
 \z:\varphi,\x:\K\varphi,\Delta}$
\quad
 \scriptsize $\left(
  \y\neq\z, \{\y,\z\} \cap
  \Var(
    \Omega\dashv\Gamma
    \Rightarrow\x:\K\varphi,\Delta
  ) = \emptyset
\right)$ \\
       \\
    $\infer[\mathsf{r}_\gpc]{\Omega \dashv \Gamma \Rightarrow \Delta}{\mathcal{S}(\gpc)(\Omega) \dashv \Gamma \Rightarrow \Delta}$\\
    \\
      \hline
    \end{tabular}
    \caption{Rules for $\K$ and GPCs.}
    \label{tab:rules2}
\end{table}

\begin{definition}
Let $\sigma$ be a sequent with induced frame
$\F_\sigma=(\W,\R)$, and let $\mathcal S$ be a $\Sigma$-system.
The \emph{closure} of $\F_\sigma$ under $\mathcal S$ is given as $C_{\mathcal S}(\F_\sigma) \coloneqq \bigl(\W,C_{\mathcal S}(\R)\bigr)$, where
\[
  C_{\mathcal S}(\R)
  \coloneqq
  \bigcup_{\wt\in\mathcal L_{\mathcal S}(\xt)}
  \R_\wt.
\]
\end{definition}

\begin{lemma}\label{lem: closure-G-frame}
Let $\gpc$ be a finite set of GPCs, let
$\mathcal S=\mathcal S(\gpc)$, and let $\sigma$ be a sequent such
that $\Var(\sigma)\neq\emptyset$.  Then
$C_{\mathcal S}(\F_\sigma)$ is a $\gpc$-frame.  Moreover, the
closure is idempotent:
\[
  C_{\mathcal S}\bigl(C_{\mathcal S}(\F_\sigma)\bigr)
  =
  C_{\mathcal S}(\F_\sigma).
\]
\end{lemma}
% \begin{proof}
%     See Appendix.
% \end{proof}
\begin{proof}[Proof sketch]
Each edge of a path in the closed relation can be expanded into a possibly empty path in the original relation; converse productions handle backward edges. Composing the corresponding grammar derivations shows that the closure satisfies the GPCs and that closing it again adds no edges. Full details are given in the arXiv version.
\end{proof}

Finally, we define the rule $\mathsf r_\gpc$ as follows.

\begin{definition}
Let $\sigma=\Omega\dashv\Gamma\Rightarrow\Delta$ be a sequent and $\gpc$ be a finite set of GPCs. The set
$\mathcal S(\gpc)(\Omega)$ is defined as follows. For all $\x,\y\in\Var$:
\begin{center}
  $\x\rel\y\in\mathcal S(\gpc)(\Omega)$ if and only if
  $\x\mathrel{C_{\mathcal S(\gpc)}(\R)}\y$ in
  $C_{\mathcal S(\gpc)}(\F_\sigma)$.
\end{center}
\end{definition}

Read bottom-up, the rule $\mathsf r_\gpc$ replaces the control
$\Omega$ of a sequent $\Omega\dashv\Gamma\Rightarrow\Delta$ by
$\mathcal S(\gpc)(\Omega)$, thereby closing it under all GPCs in
$\gpc$ in one step.

\begin{example}
(1) If $\gpc=\emptyset$, then $\mathcal L_{\mathcal S(\gpc)}(\xt)=\{\xt\}$,
% \[
%   \mathcal L_{\mathcal S(\gpc)}(\xt)=\{\xt\},
% \]
and hence $C_{\mathcal S(\gpc)}(\F_\sigma)=\F_\sigma$. (2) If $\gpc=\{\gpc(\epsilon)\}$, then $\mathcal L_{\mathcal S(\gpc)}(\xt)
  =
  \{\xt,\epsilon\}$,
% \[
%   \mathcal L_{\mathcal S(\gpc)}(\xt)
%   =
%   \{\xt,\epsilon\},
% \]
so $C_{\mathcal S(\gpc)}(\R)
  =
  \R_{\xt}\cup\R_\epsilon$.
% \[
%   C_{\mathcal S(\gpc)}(\R)
%   =
%   \R_{\xt}\cup\R_\epsilon.
% \]
Thus, read bottom-up, the rule $\mathsf r_\gpc$ does not alter the control in case (1) and adds
$\x\rel\x$ to the control for every $\x\in\Var(\sigma)$ in case (2). For a specific set $\gpc$ of GPCs, we can therefore replace $\mathsf r_\gpc$ by its specific instantiation, which simplifies the sequent calculus while it retains all properties proven for the general case. For example for (2) $\mathsf r_\gpc$ is equivalently
presented as
\begin{prooftree}
  \AxiomC{$\Omega,\{\,\x\rel\x\,\}_{\x\in\Var(\Omega\dashv\Gamma\Rightarrow\Delta)}
    \dashv\Gamma\Rightarrow\Delta$}
  \RightLabel{$\mathsf{refl}$}
  \UnaryInfC{$\Omega\dashv\Gamma\Rightarrow\Delta$}
\end{prooftree} 
\end{example}

\begin{definition}
Let $\gpc$ be a finite set of GPCs.  A \emph{proof} of a sequent
$\sigma$ in $\KW_\gpc$ is a finite tree $\pi$ labelled by sequents
according to the rules of $\KW_\gpc$ such that the root is labelled
by $\sigma$ and every leaf is labelled by an initial sequent.  We write
$\KW_\gpc\vdash\sigma$ if $\sigma$ has a $\KW_\gpc$-proof.
\end{definition}

We finish the section by proving soundness of $\KW_\gpc$.  First,
we show that all rules preserve validity over the respective
classes of frames.

\begin{lemma}\label{lem: local-soundness1}
Let $\gpc$ be a finite set of GPCs.  The rule
$\mathsf r_\gpc$ preserves validity over the class of
$\gpc$-frames.
\end{lemma}

\begin{proof}
We argue contrapositively.  Suppose that
$\sigma=\Omega\dashv\Gamma\Rightarrow\Delta$ is falsified in a
model $\fw=(\W,\R,\V)$ based on a $\gpc$-frame, under an assignment
$\lambda$.  Then $\lambda$ satisfies $\Omega$ and $\Gamma$ and
falsifies $\Delta$.  Since $\lambda$ satisfies $\Omega$, it maps
every path in $\F_\sigma$ to a path with the same label in the
frame of $\fw$.  Thus, if
$\x\rel\y\in\mathcal S(\gpc)(\Omega)$, then $\lambda(\x)\mathrel{\R_\beta}\lambda(\y)$ for some
  $\beta\in\mathcal L_{\mathcal S(\gpc)}(\xt)$.
As $\fw$ is based on a $\gpc$-frame, the finite-basis lemma and
induction on $\mathcal S(\gpc)$-derivations give
$\R_\beta\subseteq\R$.  Hence $\lambda$ satisfies
$\mathcal S(\gpc)(\Omega)$.  Since the body is unchanged, the same
model and assignment falsify the premise.
\end{proof}

\begin{lemma}\label{lem: local-soundness2}
All logical rules preserve validity over the class of all frames.
\end{lemma}
% \begin{proof}
%     See Appendix.
% \end{proof}
\begin{proof}[Proof sketch]
The propositional cases are standard, and the modal cases follow by the countermodel arguments described above. Full details are given in the arXiv version.
\end{proof}

\begin{theorem}[Soundness]
Let $\gpc$ be a finite set of GPCs and let $\sigma$ be a sequent. If
${\KW_\gpc\vdash\sigma}$, then $\sigma$ is valid over the class of $\gpc$-frames.
\end{theorem}

\begin{proof}
By induction on the height of a proof.  The initial sequents are
valid, and every induction step follows from the two local-soundness
lemmas above.
\end{proof}

\section{Completeness}\label{sec: completeness}

\noindent In this section we prove completeness of $\KW_\gpc$ for every finite
set of GPCs $\gpc$. The proof proceeds by a fair proof-search
argument. We first introduce a suitable notion of \emph{saturated sequent}.

\begin{definition}\label{def:gpc-saturation}
Let $\sigma=\Omega\dashv\Gamma\Rightarrow\Delta$ be a sequent.
The control $\Omega$ is \emph{$\gpc$-saturated in $\sigma$} if $\Omega=\mathcal S(\gpc)(\Omega)$,
where the right-hand side is evaluated in the induced frame $\F_\sigma$, and hence over the label domain $\Var(\sigma)$.
\end{definition}

If $\Var(\sigma)=\emptyset$, idempotence is immediate. Otherwise it
follows from Lemma~\ref{lem: closure-G-frame} that closure is
idempotent on the fixed label domain $\Var(\sigma)$. Since an
application of $\mathsf r_\gpc$ changes neither the body nor the label
set of a sequent, the control of its premise is $\gpc$-saturated in
that premise. Subsequent rule applications may introduce fresh labels,
in which case $\mathsf{r}_{\gpc}$ may have to be applied again.

\begin{definition}\label{def:formula-saturation}
Let $\sigma=\Omega\dashv\Gamma\Rightarrow\Delta$ be a sequent. An
occurrence of a labelled formula $\x: \varphi$ in $\Gamma$ or $\Delta$ is
\emph{formula-saturated in $\sigma$} if $\varphi = p \in \Prop$, or $\varphi = \bot$, or if $\x: \varphi$ satisfies the corresponding clause below:
\begin{enumerate}
  \item if $\x:\varphi\to\psi\in\Gamma$, then
        $\x:\varphi\in\Delta$ or $\x:\psi\in\Gamma$;
  \item if $\x:\varphi\to\psi\in\Delta$, then
        $\x:\varphi\in\Gamma$ and $\x:\psi\in\Delta$;
  \item if $\x:\K\varphi\in\Gamma$, then either
        $\y:\varphi\in\Gamma$ for every $\x\rel\y\in\Omega$, or
        $\y:\varphi\in\Delta$ for every $\x\rel\y\in\Omega$;
  \item if $\x:\K\varphi\in\Delta$, then there are distinct labels
        $\y$ and $\z$ such that $\x\rel\y,\x\rel\z\in\Omega$,
        $\y:\varphi\in\Gamma$, and $\z:\varphi\in\Delta$.
\end{enumerate}
\end{definition}

\begin{definition}
Let $\gpc$ be a finite set of GPCs. A sequent
$\sigma=\Omega\dashv\Gamma\Rightarrow\Delta$ is \emph{saturated} if $\sigma$ is not an initial sequent, its control is $\gpc$-saturated in $\sigma$ and every labelled formula in its body is formula-saturated in $\sigma$.
\end{definition}

We now define proof-search trees for $\KW_\gpc$. Such a tree records
a systematic bottom-up application of rule instances, and it need not
be finite. As shown below, if such a tree is not a proof, it contains
a finite or infinite branch from which a countermodel can be
constructed.

The source of possible infinitude is the interaction between
$\K\mathsf R$ and relational closure. A $\K\mathsf R$-application
introduces fresh labels and new relational atoms; closing the enlarged control
may then create further successors of an old label, making a
persistent left occurrence of a $\K$-formula unsaturated again. It
must consequently be processed once more. Thus the construction below
is finitely branching and fair, but not necessarily terminating. In
particular, this argument does not by itself establish the finite model
property.

Because controls and formula contexts are sets, rules may be applied
in a preserving form. A rule instance is \emph{preserving} if, for
every premise
$\Omega'\dashv\Gamma'\Rightarrow\Delta'$ of a conclusion
$\Omega\dashv\Gamma\Rightarrow\Delta$, we have (1) $\Omega\subseteq\Omega'$, (2) $\Gamma\subseteq\Gamma'$ and (3) $\Delta\subseteq\Delta'$.
Logical rules are applied preservingly by treating the principal formula also as a side formula. Every instance of $\mathsf r_\gpc$ is preserving since $\mathsf{r}_{\gpc}$ does not delete relational atoms.

\begin{definition}\label{def:proof-search}
Let $\gpc$ be a finite set of GPCs and let $\sigma$ be a sequent. A
\emph{proof-search tree} for $\sigma$ in $\KW_\gpc$ is a possibly
infinite rooted tree whose nodes are labelled by sequents and which
satisfies the following conditions.
\begin{enumerate}
  \item The root is labelled by $\sigma$.
  \item Every internal node is the conclusion of a preserving rule
        instance, and its children are exactly the premises of that
        instance.
  \item The rule $\mathsf r_\gpc$ may be applied at a sequent
        $\tau$ only if its control is not $\gpc$-saturated in $\tau$.
        A logical rule may be applied only if its principal left- or
        right-hand occurrence is not formula-saturated in its
        conclusion.
  \item A node is a leaf if and only if it is labelled by an initial sequent or
        by a saturated sequent.
  \item Every branch is \emph{fair}. Unless it terminates at an axiom,
        every formula occurrence or control that is unsaturated at a
        node is eventually \emph{discharged}: a formula occurrence either
        becomes formula-saturated or, while still unsaturated, is
        principal in a rule applied at a descendant; a control either
        becomes $\gpc$-saturated or, while still unsaturated, is
        processed by $\mathsf r_\gpc$ at a descendant.
\end{enumerate}
\end{definition}

If a formula occurrence or a control later becomes unsaturated again
because new labels or relational atoms have been introduced, the
fairness condition applies to it again. Note that proof-search trees are finitely branching. By K\H{o}nig's Lemma, infinite proof-search trees therefore contain infinite branches.

Proof-search trees can be constructed by preservingly applying rules bottom-up and keeping, along each branch, a queue of the currently unsaturated formulae and, when applicable, the control, always processing the oldest outstanding entry first. Entries that have become saturated are discarded, while newly unsaturated entries are appended to the end of the queue. To process the oldest remaining entry, apply the corresponding preserving rule; for $\K\mathsf R$, choose two distinct fresh labels, which exist because the current sequent contains only finitely many labels. Each rule application produces finitely many finite premises: logical rules add only finitely many formulae and labels, while the control in the premise of $\mathsf r_\gpc$ contains at most $|\Var(\tau)|^2$ relational atoms for a conclusion $\tau$. Repeating this construction therefore yields a finitely branching tree, although some branches may be infinite. As new entries are added at the end of the queue, every outstanding requirement either becomes saturated or is eventually processed. Hence every branch is fair. This yields the following lemma.

\begin{lemma}\label{lem:proof-search}
Let $\gpc$ be a finite set of GPCs. Every sequent $\sigma$ has a
proof-search tree in $\KW_\gpc$.
\end{lemma}

We now show that every unsuccessful proof-search yields a
countermodel.

\begin{proposition}\label{prop:proof-search}
Let $\gpc$ be a finite set of GPCs, let
$\sigma=\Omega\dashv\Gamma\Rightarrow\Delta$ be a sequent, and let
$\pi$ be a proof-search tree for $\sigma$ in $\KW_\gpc$. Then either
$\pi$ is a $\KW_\gpc$-proof of $\sigma$, or $\sigma$ is falsifiable
over the class of $\gpc$-frames.
\end{proposition}

\begin{proof}
Suppose $\pi$ is not a proof. Then $\pi$ must contain a `bad' branch: either a finite branch which ends in a saturated leaf or an infinite branch. Thus let $\rho=(\sigma_i)_{i<\ell}$ be a finite or infinite bad branch, where $\ell\in\mathbb N^+\cup\{\omega\}$, and let
$\sigma_i=\Omega_i\dashv\Gamma_i\Rightarrow\Delta_i$. We define
\[
  X_\rho=\bigcup_{i<\ell}\Var(\sigma_i),\qquad
  \Omega_\rho=\bigcup_{i<\ell}\Omega_i,\qquad
  \Gamma_\rho=\bigcup_{i<\ell}\Gamma_i,\qquad
  \Delta_\rho=\bigcup_{i<\ell}\Delta_i.
\]
Along $\rho$, preservation gives
$\Omega_i\subseteq\Omega_j$, $\Gamma_i\subseteq\Gamma_j$, and
$\Delta_i\subseteq\Delta_j$ whenever $i\leq j$.
Since $\rho$ contains no initial sequent, we also have
\begin{equation}\label{eq: completeness}
\x:\bot\notin\Gamma_\rho \text{ for every label } \x \text{ and }
\Gamma_\rho\cap\Delta_\rho=\emptyset.
\end{equation}

Indeed, a left-hand bottom
formula would already give a $\bot$-initial sequent, while persistence would
place any formula in both unions at a common node, giving an
$\mathsf{id}$-initial sequent.

We define a countermodel for $\sigma$ as follows. Choose $\star\notin X_\rho$ and let
\[
  \W_\rho=X_\rho\mathbin{\cup}\{\star\},\qquad
  \R_\rho=
  \{(\x,\y)\in X_\rho^2\mid \x\rel\y\in\Omega_\rho\}
  \cup\{(\star,\star)\}.
\]
For $\x\in X_\rho$, let
\[
  \V_\rho(\x)=\{p\in\Prop\mid \x:p\in\Gamma_\rho\},
  \qquad \V_\rho(\star)=\emptyset,
\]
and define $\lambda_\rho:\Var\to\W_\rho$ by
$\lambda_\rho(\x)=\x$ if $\x\in X_\rho$, and
$\lambda_\rho(\x)=\star$ otherwise. Finally, set
$\fw_\rho=(\W_\rho,\R_\rho,\V_\rho)$.

The point $\star$ ensures that $\W_\rho$ is non-empty and $\lambda_\rho$ total.
It is an isolated reflexive component; hence, once the relation on
$X_\rho$ is shown below to satisfy $\gpc$, adjoining $\star$ preserves
the $\gpc$-frame condition.

We now verify the frame condition. If $X_\rho=\emptyset$, the model
consists only of the reflexive point $\star$, so the claim is
immediate. Suppose $X_\rho\neq\emptyset$, fix
$\gpc(\alpha)\in\gpc$, and write
$\R_\rho^0=\R_\rho\cap X_\rho^2$. By
Lemma~\ref{lem:gpc-finite-basis}, it is enough to prove
\[
  (\R_\rho^0)_\alpha\subseteq\R_\rho^0.
\]
Suppose $(\x,\y)\in(\R_\rho^0)_\alpha$. By the definition of $\R_\rho^0$, every edge of the witnessing $\alpha$-path corresponds to a relational atom in $\Omega_\rho$, and every label on the path belongs to $X_\rho$. Since the path is finite, all its labels and relational atoms have appeared by some stage of the branch; by persistence, they therefore occur together in some $\sigma_i$. Since $\xt\mapsto_{\mathcal S(\gpc)}\alpha$, closing the control of $\sigma_i$ adds the endpoint atom $\x\rel\y$.

If $\rho$ is finite, persistence puts this witness in its terminal
sequent, whose control is saturated; hence
$\x\rel\y\in\Omega_\rho$. If $\rho$ is infinite and the endpoint
atom is not already present at $\sigma_i$, its control is
unsaturated. Fairness eventually makes that control saturated or
processes it by $\mathsf r_\gpc$; in either case
$\x\rel\y$ occurs in a later control and therefore belongs to
$\Omega_\rho$. Thus $(\R_\rho^0)_\alpha\subseteq\R_\rho^0$.
Since $\gpc(\alpha)$ was arbitrary, $(X_\rho,\R_\rho^0)$ is a
$\gpc$-frame. Together with the isolated reflexive point $\star$,
this shows that $(\W_\rho,\R_\rho)$ is a nonempty
$\gpc$-frame.

Every label occurring in $\Gamma_\rho\cup\Delta_\rho$ belongs to
$X_\rho$, and hence $\lambda_\rho(\x)=\x$. We prove simultaneously
by induction on $\varphi$ that
\[
  \x:\varphi\in\Gamma_\rho
  \Longrightarrow \fw_\rho,\x\models\varphi,
  \qquad
  \x:\varphi\in\Delta_\rho
  \Longrightarrow \fw_\rho,\x\not\models\varphi.
\]

\noindent\textsc{Case for $\varphi=\bot$.}
The left implication is vacuous by observation (\ref{eq: completeness}) while the right implication follows from the semantics of $\bot$.

\smallskip
\noindent\textsc{Case for $\varphi=p\in\Prop$.}
The left implication follows from the definition of $\V_\rho$. For
the right implication, $\x:p\in\Delta_\rho$ and observation (\ref{eq: completeness}) gives
$\x:p\notin\Gamma_\rho$; hence $p\notin\V_\rho(\x)$ and
$\fw_\rho,\x\not\models p$.

\smallskip
\noindent\textsc{Case for $\varphi=\psi\to\gamma$.}
For a finite branch, the terminal sequent is saturated; for an
infinite branch, fairness eventually saturates or processes each
persistent implication occurrence. Hence
\[
\begin{aligned}
  \x:\psi\to\gamma\in\Gamma_\rho
  &\Longrightarrow
  \x:\psi\in\Delta_\rho
  \text{ or }\x:\gamma\in\Gamma_\rho,\\
  \x:\psi\to\gamma\in\Delta_\rho
  &\Longrightarrow
  \x:\psi\in\Gamma_\rho
  \text{ and }\x:\gamma\in\Delta_\rho.
\end{aligned}
\]
In the first case, the induction hypothesis makes $\psi$ false or
$\gamma$ true at $\x$, and therefore
$\fw_\rho,\x\models\psi\to\gamma$. In the second, it makes $\psi$
true and $\gamma$ false at $\x$, and therefore
$\fw_\rho,\x\not\models\psi\to\gamma$.

\smallskip
\noindent\textsc{Case for $\varphi=\K\psi$.}
Suppose $\x:\K\psi\in\Gamma_\rho$. Every finite set $F$ of
$\Omega_\rho$-successors of $\x$ is assigned a common polarity.
Indeed, choose a node containing the modal formula and all atoms
$\x\rel\y$ with $\y\in F$. Terminal saturation in the finite case,
or fairness in the infinite case, makes the occurrence saturated or
processes it by $\K\mathsf L$. Since $F$ persists, either
$\y:\psi\in\Gamma_\rho$ for every $\y\in F$, or
$\y:\psi\in\Delta_\rho$ for every $\y\in F$.

If $\x$ has no $\Omega_\rho$-successor, both alternatives are
vacuous. Otherwise fix one successor $\y_0$. The preceding
observation puts $\y_0:\psi$ on one side; suppose it belongs to
$\Gamma_\rho$, the other case being symmetric. For any other
successor $\y$, apply the observation to $F=\{\y_0,\y\}$. Its common
polarity cannot be right-hand, since that would put
$\y_0:\psi$ in both branch unions, contradicting (\ref{eq: completeness}). Hence
\[
  \forall\y\,
  (\x\rel\y\in\Omega_\rho\Rightarrow\y:\psi\in\Gamma_\rho),
\]
or symmetrically the same statement holds with $\Delta_\rho$ in
place of $\Gamma_\rho$. The induction hypothesis therefore makes
$\psi$ uniformly true or uniformly false at all successors of
$\x$, and so $\fw_\rho,\x\models\K\psi$.

Now suppose $\x:\K\psi\in\Delta_\rho$. Terminal saturation or
fairness supplies distinct labels $\y,\z$ such that
\[
  \x\rel\y,\x\rel\z\in\Omega_\rho,\qquad
  \y:\psi\in\Gamma_\rho,\qquad
  \z:\psi\in\Delta_\rho.
\]
By the induction hypothesis, $\psi$ is true at $\y$ and false at
$\z$. Hence $\fw_\rho,\x\not\models\K\psi$.

This completes the induction. By preservation,
$\Omega\subseteq\Omega_\rho$, $\Gamma\subseteq\Gamma_\rho$, and
$\Delta\subseteq\Delta_\rho$. Hence $\fw_\rho,\lambda_\rho$
satisfy the control of the root sequent which is $\sigma$ and every formula in its antecedent, while
every formula in its consequent is false. Therefore $\fw_\rho,\lambda_\rho\not\Vdash\sigma$. Since $\fw_\rho$ is based on a $\gpc$-frame, $\sigma$ is falsifiable
over the class of $\gpc$-frames.
\end{proof}

\begin{theorem}[Completeness]\label{thm:completeness}
Let $\gpc$ be a finite set of GPCs and let $\sigma$ be a sequent. If
$\sigma$ is valid over the class of $\gpc$-frames, then
$\KW_\gpc\vdash\sigma$.
\end{theorem}

\begin{proof}
By Lemma~\ref{lem:proof-search}, choose a proof-search tree $\pi$ for
$\sigma$. Since $\sigma$ is valid, the countermodel alternative in
Proposition~\ref{prop:proof-search} is impossible. Hence $\pi$ is a
finite $\KW_\gpc$-proof of $\sigma$.
\end{proof}

\section{Further Results and Future Work}\label{sec: GL}

\noindent We shall finish the paper by announcing a few further results regarding the proof theoretic properties of our framework, as well as about $\mathsf{GL}$ with non-contingency. Because of the space limitations of the present paper, we do not develop these matters here; the full results and proofs will be presented in a future extended version.

First, we have established constructive cut elimination for all our calculi via a uniform cut admissibility proof. The proof proceeds by a standard induction on the lexicographically ordered tuple consisting of the rank of the cut formula and the combined height of the proofs of the premises, and uses appropriate versions of admissible weakening and label substitutions rules. We plan to further investigate whether we can obtain a uniform Craig interpolation result for all considered logics by employing our cut-free calculi.

Second, provability logic provides an important application of non-contingency. Under the
standard arithmetical interpretation of G\"odel--L\"ob logic, $\Box\varphi$ expresses provability in $\mathsf{PA}$, while $\K\varphi$ expresses decidability, namely that either $\varphi$ or $\neg\varphi$ is provable. 

Zolin~\cite{Zolin_2001} laid the foundation for this direction by providing sound and complete axiomatizations and cut-based sequent systems for arithmetic decidability. He conjectured that the weak transitivity axiom in his axiomatization is redundant. In ongoing work, we answer this question by deriving the weak transitivity axiom from the remaining axioms. Arithmetical completeness can be obtained in a relatively straightforward way by using the standard translation $(\K\varphi)^\Box\coloneqq \Box\varphi^\Box\vee\Box\neg\varphi^\Box$ and applying Solovay's arithmetical completeness theorem~\cite{Solovay_1976}. A conceptually stronger approach would adapt the proof of Solovay's theorem directly to the primitive non-normal language of arithmetic decidability, without passing through ordinary $\mathsf{GL}$. Developing such a direct proof is a substantial problem for future work and would be of independent interest in provability logic.

Third, building on Shamkanov's circular treatment of $\mathsf{GL}$~\cite{Shamkanov_circular_2014}, we have obtained a non-wellfounded version of our calculus for transitivity that is sound and complete for the transitive and conversely well-founded frames of $\mathsf{GL}$. We plan to study its cyclic formulation and the translation between non-wellfounded and cyclic proofs.

\vspace{2ex}
\printbibliography

\end{document}